\documentclass[letterpaper,12pt]{amsart}

\usepackage{amssymb}
\usepackage{amsmath,amscd}
\usepackage{rotating}

\numberwithin{equation}{section}

\newtheorem{lemma}{Lemma}[section]
\newtheorem{theorem}[lemma]{Theorem}
\newtheorem{proposition}[lemma]{Proposition}

\newtheorem{corollary}[lemma]{Corollary}

\theoremstyle{definition}

\newtheorem{remark}[lemma]{Remark}

\newcommand{\Gal}{\mathrm{Gal}}
\newcommand{\Hol}{\mathrm{Hol}}

\newcommand{\Aut}{\mathrm{Aut}}
\newcommand{\Perm}{\mathrm{Perm}}
\newcommand{\End}{\mathrm{End}}

\newcommand{\id}{\mathrm{id}}

\newcommand{\be}{\mathbf{e}}
\newcommand{\ff}{\mathbf{f}}
\newcommand{\bv}{\mathbf{v}}
\newcommand{\bw}{\mathbf{w}}
\newcommand{\zero}{\mathbf{0}}

\newcommand{\Z}{\mathbb{Z}}

\newcommand{\F}{\mathbb{F}}

\begin{document}
\title[Hopf-Galois Structures and Sophie Germain Primes]
{Hopf-Galois Structures on Non-Normal Extensions of Degree Related to Sophie Germain Primes}

\author{Nigel P.~Byott}
\address{Department of Mathematics, College of Engineering,
  Mathematics and Physical Sciences, University of Exeter, Exeter 
EX4 4QF U.K.}  
\email{N.P.Byott@exeter.ac.uk}

\author{Isabel Martin-Lyons}

\email{imartin.lyons@gmail.com}

\date{\today}

\bibliographystyle{amsalpha}

\begin{abstract} 
We consider Hopf-Galois structures on separable (but not necessarily
normal) field extensions $L/K$ of squarefree degree $n$. If $E/K$ is
the normal closure of $L/K$ then $G=\Gal(E/K)$ can be viewed as a
permutation group of degree $n$.  We show that $G$ has derived length
at most $4$, but that many permutation groups of squarefree degree and
of derived length $2$ cannot occur.    
We then investigate in detail the case where $n=pq$ where 
$q \geq 3$ and $p=2q+1$ are both prime. (Thus $q$ is a Sophie Germain prime and $p$ is a
safeprime). We list the permutation groups $G$ which can arise, and
we enumerate the Hopf-Galois structures for each $G$. There are six such $G$ for which the corresponding field extensions $L/K$ admit Hopf-Galois structures of both possible types. 
\end{abstract}

\maketitle

\section{Introduction}

Half a century ago, Chase and Sweedler \cite{CS} initiated Hopf-Galois theory, in part
motivated by the study of inseparable field extensions. Their work
nevertheless raises many interesting questions for separable
extensions.  Greither and Pareigis \cite{GP} showed that a separable
extension may admit many Hopf-Galois structures, and that these can be
described in group-theoretic terms.

Suppose first that $L/K$ is a field extension of degree $n$ which is
normal as well as separable, so $L/K$ is a Galois extension in the
classical sense. Let $G=\Gal(L/K)$ be its Galois group. Then $L/K$
admits at least one Hopf-Galois structure, namely that given by the
natural action of the group algebra $K[G]$. There may however be many
other Hopf-Galois structures on $L/K$. If $H$ is the $K$-Hopf algebra
acting on $L$ in one of these Hopf-Galois structures, then, for some 
group $N$ of order $n$, we have $L \otimes_K H \cong
L[N]$ as $L$-Hopf algebras. We call
the isomorphism class of $N$ the {\em type} of this Hopf-Galois
structure. The possible types, and the number of Hopf-Galois
structures of each type, depend only on the Galois group $G$.

More generally, suppose now that $L/K$ is a separable (but not
necessarily normal) extension of degree $n$. Let $E/K$ be the normal
closure of $L/K$, and let $G=\Gal(E/K)$ and $G'=\Gal(E/L)$.  Again,
each Hopf-Galois structure on $L/K$ has a type $N$, which is (the isomorphism
class of) a group of order $n$. The number of Hopf-Galois
structures of each type depends only on $G$ together with its subgroup
$G'$. Indeed, we may interpret the main result of \cite{GP} as saying
that the Hopf-Galois structures on $L/K$ depend on the isomorphism
class of $G$ as a permutation group acting transitively on the left
coset space $X=G/G'$ (or, equivalently, acting on the set of
$K$-linear embeddings of $L$ into $E$).

There is a substantial literature studying Hopf-Galois structures on
various classes of separable field extensions; see for example
\cite{kohl,Ch05,Tsang_AS} and the references therein. While some
of this work, such as \cite{Ch89, CRV, Crespo-Salguero}, treats
non-normal extensions, most of it is concerned with Galois
extensions. The study of Hopf-Galois structures on Galois extensions
has recently become of wider interest due to a connection with (skew)
braces and hence with set-theoretic solutions of the Yang-Baxter
equation (see \cite{bachiller} and the appendix to \cite{SV}).

When $n=pq$ is the product of two distinct primes, the Hopf-Galois
structures on a Galois extension of degree $n$ were determined in
\cite{pq}. Provided that $p \equiv 1 \pmod{q}$, there are two groups
of order $pq$, one cyclic and one non-abelian. A
Galois extension with either Galois group admits Hopf-Galois
structures of both types. Another case which has been investigated 
in several papers \cite{Ch03, BC, kohl13} is when $n=2pq$ for odd
primes $p$, $q$ with $p=2q+1$. (Primes $q$ such that $2q+1$ is also prime 
are named after Sophie Germain in recognition of her work on Fermat's Last Theorem
for these exponents, while primes $p$ such that $(p-1)/2$ is also prime are called 
safeprimes because of their significance in cryptography.) 
More generally, the Hopf-Galois structures
on a Galois extension of arbitrary squarefree degree $n$ were
investigated in \cite{AB-gal}. In this case, the groups of order $n$
can be classified, and given any two such groups $G$, $N$, one can
calculate the number of Hopf-Galois structures of type $N$ on a Galois
extension $L/K$ with $\Gal(L/K) \cong G$.

The question then arises of whether the results of \cite{AB-gal} can
be extended to non-normal (but separable) field extensions $L/K$ of
square-free degree $n$. The purpose of this paper is to take some
initial steps in this direction. In this situation, the type $N$ of a Hopf-Galois structure
is again a group of squarefree order $n$.
The group $G=\Gal(E/K)$, however, will have order a proper multiple of $n$ as soon as $L/K$ is not normal, and we cannot expect this order to be squarefree. 
We are nevertheless able to obtain a few general results. 
A necessary condition for $L/K$ to admit any
Hopf-Galois structures is that $G$ has derived length at most $4$
(Theorem \ref{dl4}), but this condition is very far
from sufficient (Theorem \ref{comp-dl2}).

If we wish to determine all possible Hopf-Galois structures
on extensions of a given squarefree degree $n$, we cannot proceed as in \cite{AB-gal} and start with
an arbitrary pair of groups $G$, $N$. This is because we have no {\em a priori} classification of 
permutation groups of degree $n$, and therefore we cannot specify in advance the groups $G$ we need to consider. 
A possible alternative strategy is take each group $N$ of order $n$ in
turn, and find the permutation groups $G$ which can arise from
Hopf-Galois structures of type $N$. We then have the new problem of
determining when two such groups $G_1$ and $G_2$, arising from
(possibly different) types $N_1$ and $N_2$, are isomorphic as
permutation groups. If this occurs then field extensions realising
this permutation group will admit Hopf-Galois structures of type
$N_1$ and of type $N_2$. 

It is not clear whether one can expect the above
strategy to be viable for an arbitrary squarefree $n$. Our main goal
in this paper is to carry out this strategy in the 
special case that $n=pq$ where $q \geq 3$ is a Sophie Germain prime and 
$p=2q+1$ is the associated safeprime. As there are conjecturally infinitely many
Sophie Germain primes, we expect our results to hold for infinitely
many values of $n$. For such $n$, we obtain a catalogue of the permutation
groups of degree $n$ for which the corresponding field extensions
$L/K$ admit at least one Hopf-Galois structure. We then enumerate all
Hopf-Galois structures on these extensions $L/K$, and determine 
which permutation groups $G$ are realised by Hopf-Galois structures of both
possible types. (Recall that we have two group $N$ of order $pq$, up to isomorphism.) There are six such $G$, including the two regular groups of degree $pq$, as
 shown in Table \ref{both-types}. Our catalogue by no means contains all permutation groups of 
degree $n$: those that do not occur are precisely those which cannot be realised by a
Hopf-Galois structure. 

\section{Preliminaries}

In this section, we recall some definitions and standard facts about
Hopf-Galois structures, emphasising the connection to permutation
groups. In particular, we outline the method of counting Hopf-Galois structures
developed in \cite{unique} . We also prove some technical results
which will be useful later. For a more complete account of the theory
of Hopf-Galois structures, we refer the reader to \cite[Chapter
  2]{Ch00}.

By a permutation group we mean a group $G$ together with an injective
homomorphism $\pi:G \to \Perm(X)$ into the group of
permutations of a set $X$. We will also say that $G$ acts on $X$, and,
for $g \in G$, $x \in X$, we write $g \cdot x $ for $\pi(g)(x)$. In
this paper $X$ (and hence $G$) will always be finite. The degree of
$G$ is the cardinality $|X|$ of $X$. We say that $G$ is transitive
(respectively, regular) on $X$ if, for any $x$, $y \in X$, there is
some $g \in G$ (respectively, a unique $g \in G$) with $g \cdot x
=y$. Throughout this paper, all permutation groups will be assumed transitive.

The stabiliser of $x \in X$ is by definition the subgroup $G_x=\{g \in
G: g \cdot x=x\}$. Then the stabiliser of $g \cdot x$ is $g G_x
g^{-1}$. As $G$ acts transitively on $X$ and embeds in $\Perm(X)$,
it follows that the core $\cap_{g \in G} g G_x g^{-1}$ of $G_x$ in $G$
is trivial.  Moreover $G$, with its given action on $X$, is isomorphic
as a permutation group to $G$ acting on the set of left cosets $G/G_x=\{g G_x : g \in G\}$
via its left multiplication action $\lambda: G \to \Perm(G/G_x)$, where
$\lambda(g)(hG_x)=(gh)G_x$. Thus, up to
isomorphism, a permutation group of degree $n$ can be taken to be an
abstract group $G$ acting by left multiplication on $G/G'$, where $G'$ is a subgroup of index
$n$ with trivial core. We define
$$ \Aut(G,G')=\{ \theta \in \Aut(G) : \theta(G')=G' \}. $$
Then $\Aut(G,G')$ is the group of automorphisms $\theta$ of $G$ as a
permutation group satistfying the further condition that $\theta$ fixes the distinguished element $1_G G'$ of $G/G'$,
where $1_G$ denotes the identity element of $G$. 

Now let $L/K$ be a separable field extension of finite degree $n$,
and let $E$ be its normal closure inside a fixed algebraic closure
$K^c$ of $K$.  Let $G=\Gal(E/K)$ and $G'=\Gal(E/L)$. Then $G$ acts
transitively on the $K$-linear embeddings of $L$ into $E$ (or into
$K^c$), and the stabiliser of the inclusion $L \hookrightarrow E$ is
$G'$. Thus associated to $L/K$ we have the permutation group $G$ of degree $n$ acting on
the set $X=G/G'$. 

Next let $H$ be a cocommutative $K$-Hopf algebra. We say that $L$ is an
$H$-module algebra if $H$ acts on $L$ as $K$-linear endomorphisms such
that $h \cdot(xy)=\sum_{(h)} (h_{(1)} \cdot x) (h_{(2)} \cdot y)$ for
$h \in H$ and $x$, $y \in L$, and $h \cdot k = \epsilon(h)k$ for $h
\in H$ and $k \in K$. Here we write the comultiplication $H \to H
\otimes_K H$ as $h \mapsto \sum_{(h)} h_{(1)} \otimes h_{(2)}$, and
$\epsilon:H \to K$ is the augmentation of $K$. Moreover, we say
that $L/K$ is $H$-Galois if in addition the $K$-linear map $\theta: L
\otimes_K H \to \End_K(L)$ given by $\theta(x \otimes h)(y) = x(h
\cdot y)$ is bijective. A Hopf-Galois structure on $L/K$ consists of a
cocommutative Hopf algebra $H$ and an action of $H$ on $L$ for which
$L/K$ is $H$-Galois. 

Greither and Paregis showed that the Hopf-Galois structures on $L/K$
depend only on the permutation group $G$. More precisely, the
Hopf-Galois structures correspond bijectively to the regular subgroups
$N$ of $\Perm(X)$ which are normalised by the image $\lambda(G)$ of 
left multiplication. 
For each such $N$, we have the $K$-Hopf
algebra $H=E[N]^G$, where $G$ acts on $E[G]$ simultaneously as field
automorphisms of $E$ and by conjugation (via $\lambda$) on $N$, and where
$H$ acts on $L$ by Galois descent.  We refer to the isomorphism type
of $N$ as the {\em type} of this Hopf-Galois structure. We say that
$L/K$ is {\em almost classically Galois} if the subgroup $G'$ of $G$ has a
normal complement $C$. In that case, we obtain a Hopf-Galois structure
by taking $N=C$.

Let $G$ be a permutation group, and let $G'$ be the stabiliser of a
point. We say that a separable field extension $L/K$ {\em realises} $G$ if
there is an isomorphism $\theta:G \to \Gal(E/K)$ with
$\theta(G')=\Gal(E/L)$, where again $E$ is the normal closure of
$L/K$. Moreover, we will say that $G$ is realised by a Hopf-Galois
structure of type $N$ if $L/K$ admits a Hopf-Galois structure of type
$N$. 

Given also an abstract group $N$ of order $n$, we would like to determine
the number $e(G,N)$ of Hopf-Galois structures of type $N$ on a field
extension $L/K$ which realises $G$. By the result of Greither and Pareigis,
this is the number of regular subgroups $N^*$ of $\Perm(X)$ which are
isomorphic to $N$ and normalised by $\lambda(G)$, where $X=G/G'$. The following
result simplifies the calculation of $e(G,N)$.

\begin{lemma} \cite{unique} \label{count-formula}
Let $G$, $G'$ and $N$ be as above. Let $\Hol(N)=N
\rtimes \Aut(N)$ be the holomorph of $N$, and let $e'(G,N)$ be the
number of subgroups $M$ of $\Hol(N)$ which are transitive on $N$ and isomorphic
to $G$ via an isomorphism taking the stabiliser $M'$ of $1_N$ in $M$ to $G'$. Then
$$ e(G,N) = \frac{|\Aut(G,G')|}{|\Aut(N)|} \; e'(G,N). $$
In particular, if $G$ is realised by a Hopf-Galois structure of type $N$, then
$G$ is isomorphic to a transitive subgroup $M$ of $\Hol(N)$. 
\end{lemma}

The advantage of Lemma \ref{count-formula} is that it allows us to work with $\Hol(N)$
rather than the much larger group $\Perm(X)$. 

We will write elements of $\Hol(N)=N \rtimes \Aut(N)$ as
$[\eta,\alpha]$ with $\eta \in N$ and $\alpha \in \Aut(N)$. Then the
action of $\Hol(N)$ as permutations of $N$ is given by $[\eta, \alpha]
\cdot \mu = \eta \alpha(\mu)$.  Thus the normal subgroup $N$ of $\Hol(N)$ is
identified with the group $\lambda(N)$ of left translations by $N$,
the subgroup $\Aut(N)$ is the stabiliser of 
$1_N$, and the group operation in $\Hol(N)$ is
$$ [\eta, \alpha] [\mu, \beta] =[\eta \alpha(\mu), \alpha \beta].  $$
To lighten notation, we will often write the elements $[\eta, \id_N]$
and $[1_N, \alpha]$ in $\Hol(N)$ as $\eta$, $\alpha$
respectively. Thus, for example, we have the identity
$\alpha \eta = \alpha(\eta) \alpha$
in $\Hol(N)$.

We end this section with some technical results concerning holomorphs.

\begin{proposition} \label{ab-aut}
Let $N$ be an abelian group such that $\Aut(N)$ is also abelian, and let
$A$, $A'$ be subgroups of $\Aut(N)$. Consider the subgroups $M=N
\rtimes A$ and $M'=N \rtimes A'$ of $\Hol(N)$. If
there is an isomorphism $\phi: M \to M'$ with
$\phi(N)=N$, then $M=M'$.
\end{proposition}
\begin{proof}
Let $\phi_N \in \Aut(N)$ be the restriction of $\phi$ to $N$. For $g
\in \Hol(N)$, let $C_g \in \Aut(N)$ be conjugation by $g$. Then
$C_\alpha=\alpha$ for $\alpha \in \Aut(N)$ by definition of the 
multiplication in $\Hol(N)$, and $C_\eta=\id_N$ for $\eta \in N$ since $N$ is
abelian.

Let $\alpha \in A$. Then $\phi(\alpha)=\eta \alpha'$ for some $\eta
\in N$ and $\alpha' \in A'$.  We claim that $\alpha'=\alpha$. This will show that
$M' \subseteq M$, and the same argument applied to $\phi^{-1}$ then gives equality.   

To prove the claim, let $\mu \in N$ and apply $\phi_N$ to the relation 
$C_\alpha(\mu) = \alpha \mu \alpha^{-1}$. This gives
$$ \phi_N(C_\alpha(\mu)) = \phi_N(\alpha) \phi_N(\mu)
\phi_N(\alpha)^{-1} = C_{\phi_N(\alpha)}(\phi_N(\mu)). $$ 
Hence in $\Aut(N)$ we have the equation $\phi_N C_\alpha =
C_{\phi_N(\alpha)} \phi_N$. Since $\Aut(N)$ is abelian, we conclude
that $C_\alpha=C_{\phi_N(\alpha)}$, and so
$$ \alpha= C_\alpha = C_{\phi_N(\alpha)} = C_{\eta}
C_{\alpha'}=C_{\alpha'}= \alpha'. $$ 
\end{proof}

\begin{proposition}  \label{rel-aut-char}
Let $N$ be any group and $A$ a subgroup of $\Aut(N)$. Let $M$ be the
subgroup $N \rtimes A$ of $Hol(N)$, and suppose that $N$ is characteristic in
$M$. Then the group
$$ \Aut(M,A) := \{ \theta \in \Aut(M) : \theta(A)=A\}  $$ 
is isomorphic to the normaliser of $A$ in
$\Aut(N)$. In particular, if $\Aut(N)$ is abelian then $\Aut(M,A)
\cong \Aut(N)$.
\end{proposition}
\begin{proof}
An element $\theta$ of $\Aut(M,A)$ is clearly determined by its
restrictions $\theta_N\in \Aut(N)$ and $\theta_A \in \Aut(A)$. 

Let $\eta \in N$, $\alpha \in A$, and write $\mu=\theta_N(\eta)$, $\beta=\theta_A(\alpha)$. 
Applying $\theta$ to the relation 
$\alpha \eta = \alpha(\eta) \alpha$ in $\Hol(N)$ we get
$\beta \mu=\theta_N(\alpha(\eta)) \beta$, so that
$$ \beta \mu \beta^{-1} = \theta_N(\alpha(\theta_N^{-1}(\mu))). $$
As this holds for all $\mu \in N$, and as $\beta \mu \beta^{-1} =\beta(\mu)$,
it follows that $\theta_A(\alpha)=\beta=\theta_N \alpha \theta_N^{-1}$ in $\Aut(N)$.
Hence $\theta_A$ is determined by $\theta_N$. Moreover, as $\theta_A(\alpha) \in A$ for all
$\alpha \in A$, it follows that $\theta_N$ normalises $A$.
Conversely, given any $\theta' \in
\Aut(N)$ which normalises $A$, we obtain an element of $\Aut(M,A)$ by
setting $\theta_N=\theta'$ and $\theta_A(\alpha) = \theta' \alpha
\theta'^{-1}$. This proves the first statement, and the second follows. 
\end{proof}

\section{Extensions of squarefree degree: some general results}

We first mention some old results. 
Childs \cite{Ch89} showed that a permutation group $G$ of prime degree $p$ is realised by a
Hopf-Galois structure if and only if it is soluble (so that $G$ is a subgroup of $C_p \rtimes C_{p-1}$).  More generally, let $n$ be a natural number satisfying the condition $\gcd(n, \varphi(n))=1$, where $\varphi$ is the Euler totient function. Then $n$ is squarefree and every group of order $n$ is cyclic. It follows from  \cite[Theorem 2]{unique} that if a permutation group of degree $n$ is realised 
by a Hopf-Galois structure on a field extension $L/K$, then this is the unique Hopf-Galois structure admitted by $L/K$, and $L/K$ is almost classically Galois. 

We now consider more general squarefree numbers $n$. The groups of order $n$ can be 
classified since every Sylow subgroup is cyclic. 

\begin{lemma} \label{sf-gps}
Let $N$ be a group of squarefree order $n$. Then 
$$  N= \langle \sigma,  \tau : \sigma^e=\tau^d=1, \tau \sigma=\sigma^k \tau \rangle $$ 
for some parameters $e$, $d$, $k$ such that $de=n$ and $k$ has order
$d$ mod $e$. Thus $N=C_e \rtimes C_d$ is metacyclic. 

Moreover, let
$z=\gcd(e,k-1)$ and $g=e/z$. Then $\Aut(N)$ contains a normal subgroup of
order $g$ generated by $\theta: \sigma \mapsto \sigma, \tau \mapsto
\sigma^z \tau$ and a complementary subgroup isomorphic to
$\Z_e^\times$ (the group of units in the ring of integers mod $e$)
consisting of the automorphisms $\phi_s: \sigma \mapsto \sigma^s, \tau
\mapsto \tau$ for $s \in \Z_e$. These satisfy the relations $\phi_s \theta
\phi_s^{-1}=\theta^s$. Thus $\Aut(N) \cong C_g \rtimes \Z_e^\times$ is
metabelian.
\end{lemma}
\begin{proof}
The description of $N$ is well-known, and follows from
\cite[10.1.10]{Robinson} or \cite[Lemma 3.5]{MM}.  For the description of $\Aut(N)$, see
\cite[Lemma 4.1]{AB-cyc}.
\end{proof}

\begin{corollary} \label{hol-div}
Let $N$ be a group of squarefree order $n$, and let $p$ be the largest prime factor of $n$. Then $|\Hol(N)|$
is not divisible by $p^3$. 
\end{corollary}
\begin{proof} 
With the notation of the proof of Lemma \ref{sf-gps}, we have $|\Hol(N)|=n |\Aut(N)|=ng \varphi(e)$. 
Since $\varphi(e)$ divides $\varphi(n)$ whenever $e$ divides $n$, we have that $|\Hol(N)|$ divides $n^2 \varphi(n)$.
As $p$ does not divide $\varphi(n)$, the result follows. 
\end{proof}

\begin{theorem} \label{dl4}
Let $G$ be a transitive permutation group of squarefree degree $n$
which is realised by a Hopf-Galois extension. Then $G$ has derived
length at most $4$. In particular, $G$ is soluble.
\end{theorem}
\begin{proof} 
Suppose that $G$ is realised by a Hopf-Galois structure of type $N$. Then, by Lemma \ref{sf-gps}, both $N$
and $\Aut(N)$ are metabelian, that is, they have derived length at most $2$. Hence $\Hol(N)=N \rtimes \Aut(N)$ 
has derived length at most $4$. Now Lemma \ref{count-formula} tells us that $G$ is 
isomorphic to a subgroup of $\Hol(N)$. Thus $G$ also has derived length at most $4$,
and is therefore soluble.
\end{proof}

We next show that this necessary condition for $G$ to be realised by a
Hopf-Galois structure is very far from sufficient.

\begin{theorem} \label{comp-dl2}
Let $n>6$ be a composite squarefree number. Then there is a
permutation group $G$ of degree $n$ with derived length $2$ which is
not realised by a Hopf-Galois extension.
\end{theorem}
\begin{proof}
It follows from Lemma \ref{count-formula} that any
permutation group $G$ of degree $n$ which is realised by a Hopf-Galois structure
must embed in $\Hol(N)$ for some group $N$ of order $n$.
Let $n=pm$ where $p$ is the largest prime dividing $n$. Then, 
by Corollary \ref{hol-div}, it suffices to construct a permutation group $G$ 
of degree $n$ and derived length $2$ 
whose order is divisible by $p^3$.

We first suppose $m \geq 3$.  We construct a permutation group acting
on the set $X=C_p \times C_m$ of size $pm=n$. We write $C_p$ and $C_m$ additively.
For $0 \leq j \leq m-1$,
define $\sigma_j : X \to X$ by
$$ \sigma_j(a,b) = \begin{cases} (a+1,b) & \mbox{if }b=j \\
                 (a,b) & \mbox{otherwise.} \end{cases} $$
Then $\sigma_j$ has order $p$, and $\sigma_j$, $\sigma_k$ commute for
all $j$, $k$. Hence the $\sigma_j$ generate an elementary abelian
group $V$ of order $p^m$, which has $m$ orbits on $X$, namely the sets
$C_p \times \{j\}$ for $j \in C_m$. Next define $\tau$ by
$$ \tau(a,b) = \tau(a, b+1). $$
Then $\tau$ has order $m$ and $\tau \sigma_j
\tau^{-1}=\sigma_{j+1}$. We set $G=\langle V,\tau \rangle = V \rtimes
C_m$. Then $G$ acts transitively on $X$. As $G$ is non-abelian, but has an abelian normal subgroup $V$
with $G/V \cong C_m$, we see that $G$ has derived length $2$. (In fact, $G$ is the wreath product $C_p \wr
C_m$.) Also $|G|=p^m m$ with $m \geq 3$. 

Finally, suppose that $m=2$, so that $n=2p$ for some odd prime $p$. We
interchange the roles of $p$ and $m$ in the preceding construction to
obtain a group $G$ of order $2^p p$ of derived length $2$ acting on
$C_2 \times C_p$.  If $N$ is a group of order $2p$ then $N$ is either
cyclic or dihedral. Thus $|\Hol(N)|=2p(p-1)$ or $2p^2 (p-1)$. As
$2^p > 2(p-1)$ for $p \geq 3$, $\Hol(N)$ cannot contain a subgroup of
order $2^p p$. Again, $G$ cannot be realised as a Hopf-Galois
structure.
\end{proof}

\section{Extensions of degree $pq$ with $p=2q+1$}

For the remainder of the paper, we consider Hopf-Galois structures on separable extensions of
degree $pq$, where $p=2q+1$ and $q$ are odd primes. Thus $q$ is a
Sophie Germain prime and $p$ is a safeprime.  We write $q-1=2^r s$
with $r \geq 1$ and $s$ odd. We have $\gcd(s,2pq)=1$, but we make no
further assumptions about the prime factorisation of $s$.  

Up to isomorphism, there are two groups $N$ of order $pq$, namely the
cyclic group $C_{pq}$ and the non-abelian group $C_p \rtimes C_q$.  (We adopt the convention
that the notation $A \rtimes B$ always refers to a semidirect product taken with respect to
some faithful action of $B$ on $A$.) Thus we must determine the transitive subgroups of $\Hol(C_{pq})$ and $\Hol(C_p \rtimes C_q)$. 

\subsection{Cyclic case}

Let $N$ be a cyclic group of order $pq$. We work with the presentation
$$  N= \langle \sigma, \tau : \sigma^p=\tau^q=1, \tau \sigma = \sigma
\tau \rangle. $$ 
As the subgroups $\langle \sigma \rangle$ and $\langle \tau \rangle$
are characteristic in $N$, we have 
$$ \Aut(N) \cong \Aut(\langle \sigma \rangle) \times \Aut(\langle \tau
\rangle) $$ 
where the factors are cyclic of order $p-1=2q$ and $q-1=2^r s$
respectively. Let $\alpha$, $\beta$ be automorphisms of $N$ of order $q$, $2$ respectively 
which fix $\tau$, and let $\gamma$, $\delta$ be automorphisms of order $2^r$, $s$
respectively fixing $\sigma$. Then
$\Aut(N)$ decomposes as the direct product $\langle \alpha \rangle
\times \langle \beta, \gamma \rangle \times \langle \delta \rangle$, where the 
factors have coprime orders $q$, $2^{r+1}$, $s$ respectively. A subgroup
of $\Aut(N)$ decomposes as a direct product of one subgroup from each
of these factors. The subgroups of $\langle \alpha \rangle$ are
$\langle \alpha \rangle$ and $\{\id\}$, whereas $\langle \delta
\rangle$ has one subgroup $\langle \delta^{s/d} \rangle$ of order $d$
for each divisor $d$ of $s$. We write $\sigma_0(s)$ for the number of
divisors of $s$.  The subgroups of $\langle \beta, \gamma \rangle$ are
as follows:
\begin{itemize}
\item[(i)] for $0 \leq c \leq r$, the group $\langle \beta, \gamma^{2^{r-c}}
  \rangle$ of order $2^{c+1}$, which is cyclic only when $c=0$;
\item[(ii)] for $0 \leq c \leq r$, the cyclic group $\langle
  \gamma^{2^{r-c}}\rangle$ of order $2^c$;
\item[(iii)] for $1 \leq c \leq r$, the cyclic group $\langle \beta
  \gamma^{2^{r-c}} \rangle$ of order $2^c$.
\end{itemize}

\begin{proposition}
For $1 \leq t \leq q-1$, let $J_t = \langle \sigma, [\tau, \alpha^t]
\rangle$. Then $J_t$ is a non-abelian regular subgroup of
$\Hol(N)$. Moreover, the transitive subgroups $G$ of $\Hol(N)$ are as
shown in Table \ref{cyclic-trans-subgroups}.
\end{proposition}
\begin{table}[h]
\centerline{ 
\begin{tabular}{|c|c|c|c|c|c|c|c|} \hline
  Key & Order & Parameters  & $\#$ groups & Group   \\ \hline 
  (A) & $2^{c+1} dpq^2$ & $0 \leq c \leq r$, $d \mid s$ &
  $(r+1)\sigma_0(s)$ & $N \rtimes \langle \alpha, \beta,
  \gamma^{2^{r-c}}, \delta^{s/d} \rangle$ \\ 
  (B) & $2^c dpq^2$ & $0 \leq c \leq r$, $d \mid s$ &
  $(r+1)\sigma_0(s)$ & $N \rtimes  \langle \alpha, \gamma^{2^{r-c}},
  \delta^{s/d} \rangle$ \\ 
  (C) & $2^c dpq^2$ & $1 \leq c \leq r$, $d \mid s$ & $r \sigma_0(s)$&
  $N \rtimes  \langle \alpha, \beta\gamma^{2^{r-c}}, \delta^{s/d}
  \rangle$ \\ 
  (D) & $2^{c+1} dpq$ & $0 \leq c \leq r$, $d \mid s$ &
  $(r+1)\sigma_0(s)$ & $N \rtimes \langle \beta, \gamma^{2^{r-c}},
  \delta^{s/d} \rangle$ \\ 
  (E) & $2^c dpq$ & $0 \leq c \leq r$, $d \mid s$ &
  $(r+1)\sigma_0(s)$ & $N \rtimes  \langle \gamma^{2^{r-c}},
  \delta^{s/d} \rangle$ \\ 
  (F) & $2^c dpq$ & $1 \leq c \leq r$, $d \mid s$ & $r \sigma_0(s)$ &
  $N \rtimes  \langle \beta\gamma^{2^{r-c}}, \delta^{s/d} \rangle$ \\ 
  (G) & $2pq$ & $1 \leq t \leq q-1$ & $q-1$ & $J_t \rtimes  \langle
  \beta \rangle$ \\ 
  (H) & $pq$ &  $1 \leq t \leq q-1$ & $q-1$ & $J_t$ \\
\hline 
\end{tabular}
}  
\vskip3mm

\caption{Transitive subgroups for $N$ cyclic} 
 \label{cyclic-trans-subgroups}  	
\end{table}
\begin{proof}
In $J_t$ we find that $[\tau,\alpha^t]$ has order $q$ (since $\alpha$
fixes $\tau$) and $[\tau, \alpha^t]\sigma =\alpha^t(\sigma)[\tau,
  \alpha^t]$, so $J_t$ is non-abelian of order $pq$ and regular on
$N$.

As $\Hol(N)$ contains a unique subgroup  
$H = \langle \sigma, \tau, \alpha \rangle$ 
of order $pq^2$ with index $2^r s$ coprime to $pq$, and any transitive
subgroup $M$ has order divisible by $pq$, it follows that $M \cap H$
must be transitive on $N$. Thus either $M \subset H$ or $M \cap H$ is
regular on $N$. Now the subgroups of order $pq$ in $H$ are $N$, the
groups $J_t$, and one further subgroup $\langle \sigma, \alpha \rangle$
which is not regular. We then have $M \cap H=H$ or $N$ or $J_t$ for some $t$.
In particular, every transitive subgroup $M$ contains either $N$ or some $J_t$. 

Now $N$ is normal in $\Hol(N)$, so can be extended
by any subgroup of $\Aut(N)$ to give a transitive subgroup $M$.  The
normaliser of $J_t$ in $\Aut(N)$ is $\langle \alpha, \beta \rangle$
since if $\phi \in \Aut(N)$ and $\phi(\tau) \neq \tau$, we have $\phi
[\tau, \alpha^t] \phi^{-1} = [\phi(\tau), \alpha^t] \not \in J_t$.
Hence if $M$ is a transitive subgroup containing $J_t$ but not $N$ then $M=J_t$
or $J_t \rtimes \langle \beta \rangle$. The list of transitive
subgroups then follows from the description of subgroups of $\Aut(N)$
given above.
\end{proof}

The group $N$ itself occurs in Table \ref{cyclic-trans-subgroups} as case (E) with $(c,d)=(0,1)$. We observe that, in all the subgroups in Table \ref{cyclic-trans-subgroups},
the stabiliser of $1_N$ has a normal complement (either $N$ or $J_t$), so that 
all the corresponding field extensions are almost classically Galois. The abstract isomorphism types of these groups are
as shown in Table \ref{cyclic-abst}, where the groups in cases (C) and (F) have no non-trivial direct product decomposition.
We write $D_{2m}$ for the dihedral group $C_m \rtimes C_2$ of order $2m$. 

\begin{table}[h]
\centerline{ 
\begin{tabular}{|c|c|c|c|} \hline
   Key & Restrictions & Order & Structure \\ \hline
  (A) & $(c,d) \neq (0,1)$, $(1,1)$ & $2^{c+1}pq^2 d$ & $(C_p \rtimes C_{2q}) \times (C_q \rtimes C_{2^c d})$ \\
       & $(c,d)=(1,1)$ & $4pq^2$ & $(C_p \rtimes C_{2q}) \times D_{2q}$ \\    
        & $(c,d)=(0,1)$ & $2pq^2$ & $(C_p \rtimes C_{2q}) \times C_q$ \\ \hline   
  (B) & $(c,d) \neq (0,1)$, $(1,1)$ & $2^c pq^2 d$ & $(C_p \rtimes C_q) \times (C_q \rtimes C_{2^c d})$ \\
         & $(c,d)= (1,1)$ & $2pq^2$  &$(C_p \rtimes C_q) \times D_{2q}$ \\ 
        & $(c,d) = (0,1)$ & $pq^2$ & $(C_p \rtimes C_q) \times C_q$ \\ \hline
  (C) &  & $2^c pq^2 d$ & $C_{pq} \rtimes C_{2^c d q}$ \\ \hline
  (D) & $(c,d)\neq (0,1)$, $(1,1)$ & $2^{c+1} pqd$ & $D_{2p} \times (C_q \rtimes C_{2^c d})$ \\ 
          & $(c,d)=(1,1)$ & $4pq$ & $D_{2p} \times D_{2q}$ \\
        & $(c,d)=(0,1)$ & $2pq$ & $D_{2p} \times C_q$ \\ \hline
  (E) & $(c,d) \neq (0,1)$, $(1,1)$ & $2^c pqd$ & $C_p  \times (C_q \rtimes C_{2^c d})$  \\
         & $(c,d)=(1,1)$ & $2pq$ & $C_p \times D_{2q}$ \\ 
        & $(c,d)=(0,1)$ & $pq$ & $C_{pq}$ \\ \hline
  (F) & $(c,d) \neq (1,1)$ & $2^c pqd$ &$C_{pq} \rtimes C_{2^c d}$  \\
       & $(c,d)=(1,1)$ & $2pq$ & $D_{2pq}$ \\  \hline
  (G) & & $2pq$ & $C_p \rtimes C_{2q}$ \\ \hline
  (H) & & $pq$ & $C_p \rtimes C_q$ \\
\hline 
\end{tabular}
}  
\vskip3mm

\caption{Structures of transitive subgroups for $N$ cyclic} 
 \label{cyclic-abst}  	
\end{table} 

\begin{lemma} \label{cyclic-isoms}
Of the groups in Table \ref{cyclic-trans-subgroups}, the $q-1$ groups
in case (G) are isomorphic as permutation groups, as are the $q-1$
groups in case (H).  There are no other isomorphisms, even as abstract
groups.
\end{lemma}
\begin{proof}
Each of the groups in cases (A)--(F) has the form $N \rtimes A$ for
some subgroup $A \subseteq \Aut(N)$. By Proposition \ref{ab-aut}, no
two of these groups are isomorphic. Moreover, none of these groups can
be isomorphic to a group in case (G) or (H), since the latter groups
do not contain an abelian subgroup of order $pq$. Thus we only need
consider cases (G) and (H).

Let $1 \leq t \leq q-1$ and let $\phi \in \Aut(N)$ with
$\phi(\tau)=\tau^t$. Then $\phi [\tau, \alpha^t] \phi^{-1}
=[\phi(\tau),\alpha^t]=[\tau, \alpha]^t$. Also $\phi \beta
\phi^{-1}=\beta$. Thus conjugation by $\phi$ gives an isomorphism $J_t
\rtimes \langle \beta \rangle \to J_1 \rtimes \beta$, and this is an
isomorphism of permutation groups as it fixes the stabiliser $\langle
\beta \rangle$ of $1_N$. Moreover, it restricts to an isomorphism $J_t
\to J_1$. Hence all the groups in case (G) are isomorphic as
permutation groups, and similarly for case (H).
\end{proof}

\begin{lemma}
The numbers of Hopf-Galois structures are as in Table \ref{cyclic-HGS}.
\end{lemma}

\begin{table}[h] 
\centerline{ 
\begin{tabular}{|c|c|c|c|c|c|c|c|} \hline
  Key & Order & $|\Aut(M,M')|$  & \# isom.~classes & \# HGS per  \\
      & & & & isom.~class  \\ \hline
   (A) &  $2^{c+1} dpq^2$ & $(p-1)(q-1)$ & $(r+1) \sigma_0(s)$ & $1$ \\
   (B) &  $2^c dpq^2$ & $(p-1)(q-1)$ & $(r+1) \sigma_0(s)$ & $1$ \\
   (C) &  $2^c dpq^2$ & $(p-1)(q-1)$ & $r \sigma_0(s)$ & $1$ \\
    (D) &  $2^{c+1} dpq$ & $(p-1)(q-1)$ & $(r+1) \sigma_0(s)$ & $1$ \\
   (E) &  $2^c dpq$ & $(p-1)(q-1)$ & $(r+1) \sigma_0(s)$ & $1$ \\
   (F) &  $2^c dpq$ & $(p-1)(q-1)$ & $r \sigma_0(s)$ & $1$ \\
   (G) & $2pq$ & $p-1$ & $1$ & $1$ \\
    (H) & $pq$ & $p(p-1)$ & $1$ & $p$ \\
\hline 
\end{tabular}
}  
\vskip3mm

\caption{Hopf-Galois Structures for $N$ cyclic} 
 \label{cyclic-HGS}  	
\end{table}
\begin{proof}
In each case, the stabiliser of $1_N$ in $M$ is $M'=M \cap \Aut(N)$. In
cases (A)--(F), $N$ is characteristic in $M$ since it is the unique
abelian subgroup of order $pq$.  As $\Aut(N)$ is abelian, it follows
from Proposition \ref{rel-aut-char} that $\Aut(M,M')=\Aut(N)$. Thus
$|\Aut(M,M')|=(p-1)(q-1)$.  Since there are no automorphisms between
distinct groups in cases (A)--(F), the number of isomorphism classes
is just the number of choices of parameters. (For example, in case (A)
there are $r+1$ ways to choose $c$ and $\sigma_0(s)$ ways to choose
$s$.)  For each isomorphism class $M$, the number of Hopf-Galois
structures is $|\Aut(M,M')|/|\Aut(N)|=1$ by Lemma \ref{count-formula}.

We now consider case (H). The $q-1$ regular groups $J_t$ form a single
isomorphism class.  We
have $|\Aut(J_t)|=p(p-1)$, and for $M=J_t$ we have $|M'|=1$. Thus the
number of Hopf-Galois structures of cyclic type on a $J_t$-extension
is
$$ (q-1) \; \frac{|\Aut(M,M')|}{|\Aut(N)|}  = \frac{(q-1)p(p-1)}{(p-1)(q-1)}=p.  $$  
(This calculation was already in \cite{pq}.)  

Finally we consider case (G). Let $M=J_t \rtimes \langle \beta
\rangle$. We apply Proposition
\ref{rel-aut-char} with $J_t$ in place of $N$ and $A=M' =\langle \beta \rangle $. Conjugation by $\beta$
inverts $\sigma$ and fixes the generator $\tau'=[\tau, \alpha^t]$ of
order $q$. If $\phi \in \Aut(J_t)$ we have $\phi(\sigma)=\sigma^a$ and
$\phi(\tau') = \sigma^b \tau'$ for $1 \leq a \leq p-1$ and $0 \leq b
\leq p-1$. Then $\phi$ normalises $M'$ in $\Aut(J_t)$ if and only if $b=0$. Thus
$|\Aut(M,M')|=p-1$, and the $q-1$ conjugate subgroups give
$(q-1)(p-1)/|\Aut(N)|=1$ Hopf-Galois structure of this type.
\end{proof}

We summarise the results for cyclic $N$ in the following theorem.

\begin{theorem}
There are in total $(6r+4) \sigma_0(s) +2$ isomorphism types of
permutation groups $G$ of degree $pq$ which are realised by a
Hopf-Galois structure of cyclic type. These include the two regular
groups, i.e.~the cyclic and non-abelian groups of order $pq$ (for
which the corresponding Galois extensions have $1$ and $p$ Hopf-Galois
structures of cyclic type respectively). For all the remaining groups
$G$, any field extension $L/K$ realising $G$ is almost classically Galois and 
admits a unique Hopf-Galois structure of cyclic type.
\end{theorem} 

\subsection{Metacyclic case}

Now let $N$ be the non-abelian group of order $pq$:
$$ N= \langle \sigma, \tau : \sigma^p=1=\tau^q , \tau \sigma =\sigma^g
\tau \rangle, $$ 
where $g$ has order $q$ mod $p$. (Thus $g$ is the square of a primitive root mod $p$.)
 Then, by Lemma \ref{sf-gps} (or
\cite{pq}), $\Aut(N)$ has order $p(p-1)=2pq$, and is generated by
automorphisms $\alpha$, $\beta$, $\epsilon$, of orders $q$, $2$, $p$
respectively, where
$$ \alpha(\sigma)=\sigma^g, \quad \alpha(\tau)=\tau; $$
$$ \beta(\sigma)=\sigma^{-1}, \quad \beta(\tau)=\tau; $$
$$ \epsilon(\sigma)=\sigma,  \quad \epsilon(\tau)=\sigma \tau. $$
Note that $\alpha$, $\beta$ now depend on the choice of the generator
$\tau$ of $N$. In $\Aut(N)$ we have the commutation relations
$$ \beta \alpha = \alpha \beta, \qquad \alpha \epsilon = \epsilon^g
\alpha, \qquad \beta \epsilon = \epsilon^{-1} \beta. $$ 
Thus $\Hol(N) = \langle \sigma, \tau, \alpha,\beta, \epsilon \rangle$ 
is a group of order $2p^2q^2$. 

By definition, $\Hol(N)$ is the semidirect product $N \rtimes
\Aut(N)$. It is convenient, however, to work with a different
description of $\Hol(N)$ as a semidirect product. The subgroup
$P=\langle \sigma, \epsilon \rangle \cong C_p \times C_p$ is the
unique Sylow $p$-subgroup of $\Hol(N)$, and has a complementary subgroup
$R=\langle \tau, \alpha, \beta \rangle \cong C_q \times C_q \times
C_2$. Thus we have the semidirect product decomposition $\Hol(N)=P
\rtimes R$. The element $\sigma \epsilon^{g-1} \in P$ commutes with
$\tau$:
$$  \sigma \epsilon^{g-1} \tau = \sigma \epsilon^{g-1} (\tau)
\epsilon^{g-1} = \sigma^g \tau \epsilon^{g-1} = \tau \sigma
\epsilon^{g-1}. $$ 
We write $P$ additively, and identify it with $\F_p^2$, the vector space
of dimension $2$ over the field $\F_p$ of $p$ elements. We choose basis vectors
$$\be_1=\begin{pmatrix} 1 \\ 0 \end{pmatrix}, \qquad \be_2
= \begin{pmatrix} 0 \\ 1 \end{pmatrix} $$ 
corresponding to $\sigma$, $\sigma \epsilon^{g-1}$ respectively. Then the
generators $\tau$, $\alpha$, $\beta$ of $R$ are identified with the
matrices
$$ T= \begin{pmatrix} g & 0 \\ 0 & 1 \end{pmatrix}, \qquad
A=\begin{pmatrix} g & 0 \\ 0 & g \end{pmatrix}, \qquad
B=\begin{pmatrix}  -1 & 0 \\ 0 & -1 \end{pmatrix}. $$ 
We write $I$ for the identity matrix, and $\zero$ for the identity element of $P$. For later use, we note that 
\begin{equation} \label{TA-sum}
 I + T + T^2 + \cdots + T^{q-1} = \begin{pmatrix} 0 & 0
   \\ 0 & q \end{pmatrix}, \qquad  I + A + A^2 + \cdots + A^{q-1} =
 \zero. 
\end{equation}
We write an element of $\Hol(N)=P \rtimes R$ as $[\bv,U]$ for $\bv \in
\F_p^2$ and $U$ a matrix corresponding to an element of $R$. The
multiplication in $\Hol(N)$ is then given by 
$$ [\bv, U] [\bw, V] = [\bv+U \bw, UV]. $$
As before, we often abbreviate $[\bv, I]$, $[\zero,U]$ to $\bv$, $U$.  

We next give a preliminary description of the transitive subgroups of
$\Hol(N)$ in this notation. 

\begin{lemma}  \label{metab-coarse}
A subgroup $M$ of $\Hol(N)$ is transitive on $N$ if and only it
satisfies the following two conditions: 
\begin{itemize}
\item[(i)] the image of $M$ under the quotient map $\Hol(N) \to R$ is
  one of the $2+2q$ groups  
$$ \langle T, A \rangle, \quad \langle T, A, B \rangle=R, \quad
  \langle T A^u \rangle, \quad \langle T A^u, B \rangle $$  
where $0 \leq u \leq q-1$;
\item[(ii)] $M \cap P$ is one of the three subgroups $\F_p^2$, $\F_p
  \be_1$, $\F_p \be_2$, each of which is normalised by $R$.
\end{itemize}
\end{lemma}
\begin{proof}
Suppose that $M$ is transitive. Then the orbit of $1_N$ under $M \cap
P$ must have size $p$, so the projection of $M$ into $R$ cannot be
contained in $\Aut(N)$. Hence this projection must be one of the
groups listed in (i).

If $P \not \subseteq M$ then $M \cap P$ has order $p$. Thus $M \cap P
= \F_p \be_1$ or $\F_p (\lambda \be_1+\be_2)$ for some $\lambda \in
\F_p$. In the latter case, $M$ contains an element of the form $[\bv, T
  A^a B^b]$ for some $\bv \in P$, $0 \leq a \leq q-1$, $0 \leq b \leq
1$, and
$$ T A^a B^b  (\lambda \be_1 + \be_2) = \begin{pmatrix} (-1)^b g^{a+1}
  \lambda \\ (-1)^b g^a \end{pmatrix} =(-1)^b g^a (g \lambda \be_1 +
\be_2). $$ 
The vector $g \lambda \be_1 + \be_2$ does not lie in $M \cap P$ unless $\lambda=0$. 
Thus $M \cap P$ is one of the subgroups listed in (ii). It is clear
that these are all normalised by $R$. 

Conversely, it is easy to check that if $M$ satisfies (i) and (ii)
then $M$ is transitive. 
\end{proof}

We now determine in detail the transitive subgroups $M$ of $\Hol(N)$.
If $M$ is a transitive group of even order then it is generated by its unique subgroup $M_*$ of index $2$ (which is one
of the transitive subgroups of odd order) together with an element of order $2$ which normalises $M_*$. 
 The elements of $\Hol(N)$ of order $2$ are precisely
those of the form $[\bv,B]$ for $\bv \in P$. We therefore consider the
transitive subgroups of each possible odd order in turn, in each case then identifying the subgroups
of twice that order.

\begin{proposition} \label{mc-p2q}
The transitive subgroups $M$ of $\Hol(N)$ of order divisible by $p^2 q$ are as follows:
\begin{itemize}
\item[(i)] the group $P \rtimes \langle T, A \rangle =N \rtimes \langle \alpha, \epsilon \rangle$ of order $p^2q^2$;
\item[(ii)] the group $P \rtimes \langle T, A, B \rangle =\Hol(N)$ of order $2 p^2q^2$;
\item[(iii)] the $q$ groups $P \rtimes \langle TA^u \rangle$ of order $p^2 q$ for $0 \leq u \leq q-1$;
\item[(iv)] the $q$ groups $P \rtimes \langle TA^u, B \rangle$ of order $2 p^2 q$ for $0 \leq u \leq q-1$.
\end{itemize}
\end{proposition}
\begin{proof}
This follows immediately from Lemma \ref{metab-coarse}.
\end{proof}

\begin{proposition} \label{mc-pq2}
The transitive subgroups $M$ of $\Hol(N)$ of order $pq^2$ are as follows:
\begin{itemize}
\item[(i)] $\langle \be_1, T, [\mu \be_2,A] \rangle$ for $\mu \in \F_p$;
\item[(ii)]   $\langle \be_2, [\mu \be_1, T], [\mu \be_1,A] \rangle$ for $\mu \in \F_p$.
\end{itemize}
Each of these $2p$ groups is non-abelian with structure $C_q \times (C_p \rtimes C_q)$. 
\end{proposition}
\begin{proof}
By Lemma \ref{metab-coarse}, if $M$ is a transitive subgroup of order $pq^2$ then 
$M \cap P=\F_p \be_1$ or $\F_p \be_2$, and the projection of $M$
to $R$ is $\langle T, A \rangle$. 

If $M \cap P= \F_p \be_1$ then $M$ is generated by elements
$\be_1$, $[\bv, T]$, $[\bw, A]$
for some $\bv$, $\bw \in \F_p^2$. Replacing $[\bv, T]$ by $(r\be_1) [\bv, T]$ for a suitable $r$, we may assume that $\bv =\lambda
\be_2$ for some $\lambda \in \F_p$.  
Similarly, we may assume that $\bw =\mu \be_2$ for some $\mu \in \F_p$. 
Now, using (\ref{TA-sum}), we have  
$$  [\lambda \be_2, T]^q =[(I+T+T^2+ \cdots+ T^{q-1})\lambda
  \be_2,T^q]=q \lambda \be_2. $$  
As $\be_2\not \in M$, it follows that $\lambda=0$. However, $[\mu
  \be_2,A]^q=\zero$ for any choice of $\mu$. 
Then $M$ is generated by $\be_1$, $T$ and $[\mu \be_2, A]$, which have
orders $p$, $q$, $q$ respectively.  
We check that  $T$ and $[\mu \be_2,A]$ commute, and that 
$$ T \be_1 T^{-1} = g \be_1  = [\mu \be_2,A] \be_1 [\mu \be_2,
  A]^{-1}. $$ 
Thus $\be_1$, $T$ and $[\mu \be_2, A]$ do indeed generate a group of order $pq^2$, 
and its centre is generated by the element $[\mu \be_2, AT^{-1}]$ of order $q$.
Thus $M \cong C_q \times (C_p \rtimes C_q)$. This gives the $p$ groups in (i). 

If now $M \cap P = \F_p \be_2$ then $M$ is generated by elements 
$\be_2$, $[\lambda \be_1, T]$, $[\mu \be_1, A]$
for $\lambda$, $\mu \in \F_p$.  In this case,
the second and third generators have order $q$ for any choice of
$\lambda$, $\mu$.  We calculate 
$$ [\lambda \be_1, T] [\mu \be_1, A] [\lambda \be_1,T]^{-1} =
[(\mu-\lambda)(g-1) \be_1, I] [\mu \be_1, A]. $$
Thus $(\mu-\lambda)(g-1) \be_1 \in M$, so $\lambda=\mu$ and
$M$ is as in (ii).  We then find that $[\mu \be_1, T]$ is in the centre of $M$, while
$$ [\mu \be_2, A] \be_2 [\mu \be_2, A]^{-1} = g\be_2. $$ 
This gives the $p$ groups in (ii). 
\end{proof}

\begin{proposition} \label{mc-2pq2}
The transitive subgroups $M$ of $\Hol(N)$ of order $2pq^2$ are as follows:
\begin{itemize}
\item[(i)] the $p$ groups $\langle \be_1, T, [\mu \be_2,A], [2(1-g)^{-1} \mu \be_2,B]  \rangle$ for $\mu \in \F_p$;
\item[(ii)] the $p$ groups $\langle \be_2, [\mu \be_1, T], [\mu \be_1,A], [2(1-g)^{-1} \mu \be_1, B] \rangle$ for $\mu \in \F_p$.
\end{itemize}
These groups all have the structure $C_q \times (C_p \rtimes C_{2q})$.
\end{proposition}
\begin{proof}
Any such group is obtained by adjoining an element of order $2$ to one of 
the groups in Proposition \ref{mc-pq2}(i) or (ii). In the first case, we 
may take the generator of order $2$ to be $[\nu \be_2,B]$ for some $\nu \in
\F_p$. Then
$$ M =  \langle \be_1, T, [\mu \be_2, A], [\nu \be_2, B] \rangle.  $$ 
We require $[\nu \be_2, B]$ to normalise the subgroup of order $pq^2$. 
We calculate
\begin{eqnarray}
  [\nu \be_2, B][\mu \be_2, A][\nu \be_2, B]^{-1} & = & [\nu \be_2 -
    \mu \be_2, BA][ \nu \be_2, B] \nonumber \\ 
  & = &  [\nu \be_2 - \mu \be_2 -g \nu \be_2, A] \label{A-conj-B} 
\end{eqnarray}  
which lies in $M$ only if $\nu - \mu - g\nu=\mu$, so that
$\nu=2(1-g)^{-1} \mu$, and then $[\mu \be_2,A]$ and $[\nu \be_2, B]$ commute. Similar (but simpler) calculations show that
$[\nu \be_2,B]$ commutes with $T$ and inverts
$\be_1$.  Thus $M \cong C_q \times (C_p \rtimes C_{2q})$, and we obtain the groups in (i). 

If the subgroup of $M$ of index 2 is as in Proposition \ref{mc-pq2}(ii), 
we have
$$ M=\langle \be_2, [\mu \be_1, T],  [\mu \be_1, A], [\nu \be_1, B] \rangle $$ 
for some $\nu \in \F_p$. Consideration of $[\nu \be_1, B][\mu \be_1,
  A] [\nu \be_1,B]$ shows that again $\nu=2(1-g)^{-1} \mu$, and then $[\nu
  \be_1, B]$ and $[\mu \be_1, A]$ commute. Similarly $[\nu \be_1, B]$
and $[\mu \be_1, T]$ commute, and conjugation by $[\nu \be_1,B]$ inverts
$\be_2$.  We then get the groups in (ii).
\end{proof}

\begin{proposition} \label{mc-pq}
The transitive subgroups $M$ of $\Hol(N)$ of order $pq$ are as follows:
\begin{itemize}
\item[(i)] the $p(q-2)$ non-abelian groups $\langle \be_1, [\lambda \be_2, TA^u] \rangle$ for
$\lambda \in \F_p$ and $1 \leq u \leq q-2$;
\item[(ii)] the $p$ cyclic groups $\langle \be_1, [\lambda \be_2, TA^{-1}] \rangle$ for
$\lambda \in \F_p$;
\item[(iii)] the non-abelian group $\langle \be_1, T\rangle = N$;
\item[(iv)] the $p(q-2)$ non-abelian groups $\langle \be_2, [\lambda \be_1, TA^u] \rangle$ for
$\lambda \in \F_p$ and $1 \leq u \leq q-2$;
\item[(v)] the $p$ cyclic groups $\langle \be_2, [\lambda \be_1, T] \rangle$ for
$\lambda \in \F_p$;
\item[(vi)] the non-abelian group $\langle \be_2, TA^{-1} \rangle$.
\end{itemize}
\end{proposition}
\begin{proof}
The transitive subgroups of order $pq$ are regular, and were already found in \cite{pq}.
For completeness, we determine them here. 
Any such group $M$ contains either $\be_1$ or $\be_2$. 

If $M$ contains $\be_1$ then
$M=\langle \be_1, [\lambda \be_2, TA^u]  \rangle$
for $\lambda \in \F_p$ and $0 \leq u \leq q-1$. As 
$$ TA^u = \begin{pmatrix} g^{u+1} & 0 \\ 0 & g^u \end{pmatrix}, $$
we have 
$$ [\lambda \be_2, TA^u]^q = \begin{cases} \zero & \mbox{if } u
  \neq 0, \\ 
   q \lambda \be_2 & \mbox{if }u=0. \end{cases}. $$ 
Hence if $u=0$ we must have $\lambda=0$. In either case, $[\lambda
  \be_2, TA^u]$ has order $q$. Moreover, we have
$$ [\lambda \be_2, TA^u]\be_1 [\lambda \be_2, TA^u]^{-1} =
g^{u+1} \be_1. $$ 
Thus $M$ is abelian if and only if $u=q-1$. Taking $1 \leq u \leq
q-2$, and any $\lambda$, we get the $p(q-2)$ groups in (i).  
Taking  $u=q-1$ and $\lambda$ arbitrary, we get the $p$ groups in (ii).
Taking $u=0$ and $\lambda=0$, we get the group in (iii). 

Similarly, a transitive subgroup of order $pq$ containing $\be_2$ is of the form
$M= \langle \be_2, [\lambda\be_1,TA^u] \rangle$. 
We have $[\lambda \be_1, TA^u]^q=q \lambda \be_1$ if $u=q-1$, so we
must have $\lambda=0$ in this case. The two generators commute if and
only if $u=0$. Thus we get the groups in (iv), (v) and (vi).  
\end{proof}

\begin{proposition} \label{mc-2pq}
The transitive subgroups $M$ of $\Hol(N)$ of order $2pq$ are as follows:
\begin{itemize}
\item[(i)] the $p(q-2)$ groups $\langle \be_1, [\lambda \be_2, TA^u], [2(1-g^u)^{-1} \lambda \be_2,B] \rangle$ of type $C_p \rtimes C_{2q}$ for
$\lambda \in \F_p$ and $1 \leq u \leq q-2$;
\item[(ii)] the $p$ groups $\langle \be_1, [\lambda \be_2, TA^{-1}], [2(1-g^{-1})^{-1} \lambda \be_2,B] \rangle$ of type $D_{2p} \times C_q$ for
$\lambda \in \F_p$;
\item[(iii)] the $p$ groups $\langle \be_1, T, [\mu \be_2,B] \rangle$  of type $C_p \rtimes C_{2q}$ for $\mu \in \F_p$;
\item[(iv)] the $p(q-2)$ groups $\langle \be_2, [\lambda \be_1, TA^u], [2(1-g^{u+1})^{-1} \lambda \be_2, B] \rangle$ of type $C_p \rtimes C_{2q}$ for
$\lambda \in \F_p$ and $1 \leq u \leq q-2$;
\item[(v)] the $p$ groups $\langle \be_2, [\lambda \be_1, T], [2(1-g)^{-1} \lambda \be_2,B] \rangle$ of type $D_{2p} \times C_q$ for
$\lambda \in \F_p$;
\item[(vi)] the $p$ groups $\langle \be_2, TA^{-1}, [\mu \be_2,B] \rangle$ of type $C_p \rtimes C_{2q}$ for
$\mu \in \F_p$.
\end{itemize}
\end{proposition}
\begin{proof}
First suppose that $\be_1 \in M$, so $M$ contains one of the subgroups in Proposition \ref{mc-pq}(i), (ii) or (iii). Then
$$ M=\langle \be_1,  [\lambda \be_2, TA^u], [\mu \be_2, B]
\rangle $$ 
for some $\mu \in \F_p$, where $\lambda=0$ if $u=0$, and the first two
generators commute if and only if $u=q-1$.  A calculation similar to
(\ref{A-conj-B}) gives
$$ [\mu \be_2, B] [\lambda \be_2, TA^u]  [\mu \be_2, B] = [(\mu
  -\lambda-g^u \mu) \be_2, TA^u], $$ 
so that $\mu - \lambda-g^u \mu = \lambda$. Then $[\mu \be_2,B]$ commutes with
$[\lambda \be_2, TA^u]$ while $[\mu \be_2, B]
  \be_1 [\mu \be_2, B]=-\be_1$.  
If $u \neq 0$, we have $\mu = 2(1-g^u)^{-1} \lambda$, and the
resulting group is of type $C_p \rtimes C_{2q}$ if $u \neq q-1$ and
$D_{2pq}$ if $u=q-1$. This gives the $p(q-2)$ groups in (i) and the
$p$ groups in (ii). If $u=0$, we have $\lambda=0$ but we may choose
$\mu$ arbitrarily.  Thus we obtain the $p$ groups in (iii).

If $\be_2 \in M$, similar calculations give cases (iv), (v) and (vi).
\end{proof}

The results of Propositions \ref{mc-p2q}--\ref{mc-2pq} are summarised in
Table \ref{metab-trans-subgroups}. 

\begin{table}[h]
\centerline{ 
\begin{tabular}{|c|c|c|c|c|c|c|c|} \hline
   Order & Parameters  & Structure & Group   \\ \hline 
    $p^2 q^2$ & & $\F_p^2 \rtimes (C_q \times C_q)$ &  $P \rtimes
  \langle T, A\rangle = N \rtimes \langle \alpha, \epsilon \rangle$
  \\  \hline
    $2 p^2 q^2 $ & & $\F_p^2 \rtimes(C_q \times C_{2q})$   & $\Hol(N)$ \\ \hline
    $p^2 q$ & $0 \leq u \leq q-1$ & $\F_p^2 \rtimes C_q$ & $P \rtimes
  \langle T A^u \rangle$ \\ \hline
   $2p^2 q$ & $0 \leq u \leq q-1$ & $\F_p^2 \rtimes C_{2q}$  & $P \rtimes \langle
  TA^u, B \rangle$ \\ \hline
   $pq^2$ & $\mu \in \F_p$ & $C_q \times (C_p \rtimes C_q)$ &
  $\langle \be_1,T, [\mu \be_2, A] \rangle$ \\ 
     & $\mu \in \F_p$ & $C_q \times (C_p \rtimes C_q)$ &
       $\langle \be_2,  [\mu \be_1,T], [\mu \be_1, A] \rangle$ \\ \hline
    $2pq^2$ & $\mu \in \F_p$ & $C_q \times (C_p \rtimes C_{2q})$ & $\langle \be_1,
       T, [\mu \be_2, A], [2(1-g)^{-1} \mu \be_2, B] \rangle$
       \\   
    & $\mu \in \F_p$ & $C_q \times (C_p \times C_{2q})$ & $\langle \be_2, [\mu
         \be_1,T], [\mu \be_1, A], [2(1-g)^{-1} \mu \be_1, B] \rangle$
       \\   \hline
    $pq$ & $1 \leq u \leq q-2$, $\lambda \in \F_p$ & $C_p
       \rtimes C_q$ & $\langle \be_1, [\lambda \be_2, TA^u] \rangle$
       \\
    & $\lambda \in \F_p$ & $C_{pq}$ & $\langle \be_1,
         [\lambda \be_2, TA^{-1}] \rangle$ \\ 
    & & $C_p \rtimes C_q$ & $\langle \be_1, T \rangle = N$
        \\ 
     & $1 \leq u \leq q-2$, $\lambda \in \F_p$ & $C_p
           \rtimes C_q$ &  $\langle \be_2, [\lambda \be_1, TA^u]
           \rangle$ \\ 
     & $\lambda \in \F_p$ & $C_{pq}$ & $\langle \be_2,
                  [\lambda \be_1, T] \rangle$  \\ 
     & & $C_p \rtimes C_q$ &  $\langle \be_2, TA^{-1} \rangle$ \\ \hline
    $2pq$ & $1 \leq u \leq q-2$, $\lambda \in \F_p$ &  $C_p
       \rtimes C_{2q}$ &
       $\langle \be_1, [\lambda \be_2, TA^u], [2(1-g^u)^{-1} \lambda
         \be_2, B] \rangle$ \\ 
      & $\lambda \in \F_p$ & $D_{2p} \times C_q$ & $\langle \be_1, [\lambda
           \be_2, TA^{-1}], [2(1-g^{-1})^{-1} \lambda \be_2, B]
         \rangle$ \\ 
     & $\mu \in \F_p$ & $C_p \rtimes C_{2q}$ & $\langle \be_1, T,
           [\mu \be_2, B] \rangle$ \\ 
      & $1 \leq u \leq q-2$, $\lambda \in \F_p$ &  $C_p \rtimes C_{2q}$  &
           $\langle \be_2, [\lambda \be_1, TA^u], [2(1-g^{u+1})^{-1}
             \lambda \be_1, B] \rangle$  \\           
     & $\lambda \in \F_p$ &  $D_{2p} \times C_q$ & $\langle \be_2, [\lambda
                    \be_1, T], [2(1-g)^{-1} \lambda \be_1, B] \rangle$ \\                     
     & $\mu \in \F_p$ &  $C_p \rtimes C_{2q}$ & $\langle \be_1, 
                    TA^{-1}, [\mu \be_2, B] \rangle$ \\ 
\hline
\end{tabular}
}  
\vskip3mm

\caption{Transitive subgroups for $N$ metabelian} 
 \label{metab-trans-subgroups}  	
\end{table}

For each of the transitive subgroups $M$ in Table \ref{metab-trans-subgroups}, we wish to compute
$|\Aut(M,M')|$, where $M'$ is the stabiliser of $1_N$ in $M$. We also
want to know when any two of these groups $M$ are isomorphic, either
as abstract groups or as permutation groups. We observe that $P \cap \Aut(N)
= \langle \epsilon \rangle$.  This
subgroup is also generated by $\epsilon^{1-g} = \sigma (\sigma \epsilon^{g-1})^{-1}$, which corresponds to the vector
$$ \ff = \begin{pmatrix} 1 \\ -1 \end{pmatrix}. $$

\begin{proposition} \label{Aut-p2q2} 
For the two groups $M$ in Proposition \ref{mc-p2q}(i) and (ii), we have 
 $|\Aut(M,M')|=2p(p-1)$. 
\end{proposition}
\begin{proof}
First let $M= P \rtimes \langle T, A\rangle$ be the group of order $p^2 q^2$ 
in Proposition \ref{mc-p2q}(i). Then $M$ has exactly two normal subgroups
$\F_p \be_1$ and $\F_p \be_2$ of order $p$, so any $\theta \in
\Aut(M,M')$ must either preserve or swap these two subgroups. Thus
the restriction of $\theta$ to $P$ acts as a matrix of the form
$$ \begin{pmatrix} x & 0 \\ 0 & y \end{pmatrix} \mbox{ or
} \begin{pmatrix} 0 & x \\ y & 0 \end{pmatrix} $$ 
with $x$, $y \in \F_p^\times$.  However, $\theta$ must also fix
the group $M'=M \cap \Aut(N)$ of order $pq$, and hence must fix its unique Sylow $p$-subgroup. Hence $\theta(\ff)  \in \F_p \ff$.  Since
$$ \begin{pmatrix} x & 0 \\ 0 & y \end{pmatrix} \ff = \begin{pmatrix}
  x \\ -y \end{pmatrix} \mbox{ and }  
    \begin{pmatrix} 0 & x \\ y & 0 \end{pmatrix} \ff = \begin{pmatrix}
      -x \\ y \end{pmatrix}, $$ 
we have $y=x$ in either case. Let us define $\pi$ to be the identity
map on $\{1,2\}$ in the first case, and the unique transposition on
$\{1,2\}$ in the second, 
so that $\theta(\be_i)=x \be_{\pi(i)}$ in either case. 

Now the matrix $A$ corresponds to $\alpha \in \Aut(N)$ of order $q$, 
so
$$ \theta(A) = [\lambda \ff, A^c] $$
with $\lambda \in \F_p$ and $1 \leq c \leq q-1$. From the relations $A \be_i A^{-1}=g\be_i$ for $i=1$, $2$ we obtain 
$$ [\lambda \ff, A^c] \be_{\pi(i)} [\lambda \ff,A^c]^{-1} = g \be_{\pi(i)}, $$
which holds if and only if $g^c \be_{\pi(i)} =  g \be_{\pi(i)}$. Hence $c=1$. 
Also, we must have
$$ \theta(T)=[\bv, T^d A^m ] $$
for some vector $\bv$ and some $d$, $m \in \F_q$ with $d \neq 0$.
As $T$ and $A$ commute, we have
$$ [\lambda \ff, A] [\bv, T^d A^m]= [\bv, T^d A^m] [\lambda \ff, A], $$
which is equivalent to
$$ \lambda \ff +g \bv = \bv + \lambda T^d A^m \ff, $$
and so to
\begin{equation} \label{p2q2-v}
   \bv = (g-1)^{-1} \lambda (T^d A^m-I) \ff. 
\end{equation}
Thus $\bv$ is determined by $\lambda$, $m$ and $d$.  If $\pi$ is the
identity map, the relations 
\begin{equation} \label{p2q2eqn}
   T \be_1 T^{-1}=g\be_1, \qquad T \be_2 T^{-1}=\be_2  
\end{equation}
give
$$ [\bv, T^d A^m] x \be_1 [\bv, T^d A^m]^{-1} = gx\be_1, \qquad  [\bv,
  T^d A^m] x \be_2 [\bv, T^d A^m]^{-1} = x\be_2, $$ 
which are equivalent to  
$$ g^{m+d} x \be_1 = gx \be_1, \qquad g^m x \be_2 = x \be_2, $$
so that $m=0$, $d=1$.  If $\pi$ is the transposition, then
(\ref{p2q2eqn}) gives   
$$ [\bv, T^d A^m] x \be_2 [\bv, T^d A^m]^{-1} = gx\be_2, \qquad  [\bv,
  T^d A^m] x \be_1 [\bv, T^d A^m]^{-1} = x\be_1, $$ 
so that 
$$ g^m x \be_2 = gx \be_2, \qquad g^{d+m} x \be_1 = x \be_1, $$
and $m=1$, $d=q-1$. Hence for $\theta \in \Aut(M, M')$ we have either
$$ \theta(\be_1)=x \be_1, \quad \theta(\be_2)=x\be_2, \quad
\theta(A)=[\lambda \ff, A], \quad \theta(T) = [\bv,T] $$ 
with $\bv  = \lambda \be_1$, or
$$ \theta(\be_1)=x \be_2, \quad \theta(\be_2)=x\be_1, \quad
\theta(A)=[\lambda \ff, A], \quad \theta(T) =
      [\bv,T^{q-1} A] $$ 
with $\bv  = - \lambda \be_2$, 
where $x \in \F_p^\times$ and $\lambda \in \F_p$. Hence $|\Aut(M,M')|=2p(p-1)$. 

We next consider the case where $M$ is the full group $\Hol(N)$ of
order $2p^2q^2$ in Proposition \ref{mc-p2q}(ii). If $\theta \in \Hol(M,M')$ then the restriction of
$\theta$ to the unique subgroup of index $2$ must be one of the
$2p(p-1)$ automorphisms just described, and
$\theta(B) \in \Aut(N)$ has order
$2$, so that $\theta(B) = [\mu \ff, B]$
for some $\mu \in \F_p$. Since $\theta(A)$ and
$\theta(B)$ must commute, we have
$$ \lambda \ff + g \mu \ff = \mu \ff - \lambda \ff, $$
so
\begin{equation} \label{p2q2-mu}
     \mu = 2(1-g)^{-1} \lambda. 
\end{equation}
Finally, as  $\theta(T)$ and $\theta(B)$ must
commute, we require 
$$ [\mu \ff, B] [\bv, T^d A^m] =   [\bv, T^d A^m]  [\mu \ff, B], $$
which is equivalent to 
$$ \mu \ff - \bv = \bv + \mu T^d A^m \ff. $$
By (\ref{p2q2-v}) and (\ref{p2q2-mu}), this last condition is
automatically satisfied. Thus each of the $2p(p-1)$ automorphisms
$\theta$ of the subgroup of $M$ of index $2$ extends uniquely to an automorphism of 
$M$. Again we have $|\Aut(M,M')|=2p(p-1)$.
\end{proof}

\begin{proposition} \label{Aut-p2q}
Let
$$ M_u = P \rtimes \langle TA^u \rangle, \qquad 0 \leq u \leq q-1  $$
be the $q$ groups of order $p^2 q$ in Proposition \ref{mc-p2q}(iii), and let
$$ \widehat{M}_u = P \rtimes \langle TA^u, B \rangle, \qquad 0 \leq u \leq q-1  $$
be the $q$ groups of order $2p^2 q$ in Proposition \ref{mc-p2q}(iv).
Then $M_u$ and $M_v$ are isomorphic (as abstract groups or permutation groups)
if and only if $u+v+1=0 \pmod{q}$. Hence the groups $M_u$ fall into $\frac{1}{2}(q+1)$ isomorphism classes.
Similarly for the groups $\widehat{M}_u$. Moreover
$$ |\Aut(M_u,M_u')| = \begin{cases} p^2 (p-1) & \mbox{if } u \neq \frac{1}{2}(q-1), \\
       2p^2 (p-1) & \mbox{if } u =\frac{1}{2}(q-1) \end{cases} $$
and 
$$ |\Aut(\widehat{M}_u,\widehat{M}_u')| = \begin{cases} p (p-1) & \mbox{if } u \neq \frac{1}{2}(q-1), \\
       2p (p-1) & \mbox{if } u =\frac{1}{2}(q-1). \end{cases} $$
\end{proposition}
\begin{proof}
The generator $TA^u$ of $M_u$ acts on $P$ with 
two distinct eigenvalues $g^{u+1}$, $g^u$.  Any element of order $q$
in $M_u$ has the form $[\bv, (TA^u)^c]$ for some $\bv \in P$ and
$c \in \F_q^\times$, and this element acts on $P$
with eigenvalues $g^{(u+1)c}$, $g^{uc}$. Thus if $M_u$ and $M_v$ are
isomorphic as abstract groups, then $\{g^{(u+1)c},g^{uc}\} = \{
g^{v+1}, g^v\}$ for some $c$. Thus either
$(u+1)c = v+1$, $uc=v$ in $\F_q$, or 
$(u+1)c=v$, $uc=v+1$. In the first case $c=1$ so $u=v$. In the second case
$c=-1$ so $u+v+1=0$, and then $v \neq u$ unless
$u=\frac{1}{2}(q-1)$.  We then have an isomorphism
$\phi: M_u \to M_v$ given by $\phi(\be_1)=\be_2$, $\phi(\be_2)=\be_1$,
$\phi(T A^u) = (T A^v)^{-1}$.  As $\phi(\ff)=-\ff$, it
follows that $\phi$ is an isomorphism of permutation groups. Thus the
groups $M_u$ fall into $\frac{1}{2}(q+1)$ isomorphism classes.

The argument for the groups $\widehat{M}_u$ is similar, 
using the fact that the element $TA^u B$ of order $2q$ acts with eigenvalues
$-g^{u+1}$, $-g^u$.

Now let $\theta \in \Aut(M_u,M_u')$. Again, 
$\theta$ must either fix or swap the two eigenspaces $\F_p \be_1$,
$\F_p \be_2$, and must also fix the stabiliser $\F_p \ff$ of the
identity element of $N$. Hence
$$ \theta(\be_i) = x \be_{\pi(i)} \mbox{ for }i=1, 2, $$
where $x \in \F_p^\times$ and $\pi$ is either the identity map or the
transposition on $\{1,2\}$. Moreover $\theta(TA^u)=[\bv,
  (TA^u)^c]$ for some $\bv \in P$ and some $c \neq 0$.  If $\pi$ is the identity map, the
commutation relations for $\be_1$ and $\be_2$ reduce to
$g^{(u+1)c}=g^{u+1}$, $g^{uc}=g^u$, so that $c=1$ and we obtain no
restriction on $\bv$.  If $\pi$ is the transposition, we obtain
$g^{(u+1)c}=g^u$, $g^{uc}=g^{u+1}$, so that $c=-1$ and $u=(q-1)/2$,
again with no restriction on $\bv$.  Thus we have
$$ |\Aut(M_u,M_u')| = \begin{cases} p^2 (p-1) & \mbox{if } u \neq \frac{1}{2}(q-1), \\
       2p^2 (p-1) & \mbox{if } u = \frac{1}{2}(q-1). \end{cases} $$
       
Finally, let $\theta \in \Aut(\widehat{M}_u,\widehat{M}_u')$. Arguing as above, we have
$\theta(TA^u B)=[\bv,(TA^u B)^c]$, where $c= \pm 1$ if $u=-(q-1)/2$ and $c=1$ otherwise. 
However, since $(TA^u B)^q=B$ lies in the stabiliser $\langle  \ff, B \rangle$ of the identity element of $N$, we have the additional constraint that $[\bv,(TA^u B)^c]^q$ also lies in this group. Writing $S=(TA^u B)^c$, this means that $(I+S+ \cdots+S^{q-1}) \bv \in \F_p \ff$. Now $S$ is a diagonal matrix whose diagonal entries have order $2q$ or $2$ in $\F_p^\times$. It follows that the matrix $I+S+ \cdots + S^{q-1}$ is invertible, so that there are only $p$ possibilities for $\bv$. Hence 
$|\Aut(\widehat{M}_u,\widehat{M}_u')|= p^{-1} |\Aut(M_u,M_u')|$. 
\end{proof}

\begin{remark} \label{Mu}
We describe in more detail the structure of the groups $M_u$ and $\widehat{M}_u$, all of which 
were denoted rather loosely in Table \ref{metab-trans-subgroups} by $\F_p^2 \rtimes C_q$ or 
 $\F_p^2 \rtimes C_{2q}$. We need only consider 
$0 \leq u \leq \frac{1}{2}(q-1)$. When $u=0$, the generator $TA^u=T$ commutes with $\be_2$, so 
$M_0 \cong C_p \times (C_p \rtimes C_q)$. Thus $M_0$ contains a normal subgroup of order $pq$, which is a 
complement to $M_0'=\F_p \ff$. Similarly, the group $\widehat{M}_0$
contains a normal complement to $\widehat{M}_0'$. When $1 \leq u \leq \frac{1}{2}(q-1)$, the generator $TA^u$ 
acts on $P$ with the two distinct eigenvalues $g^u$, $g^{u+1}$ of order $q$. (In the exceptional case $u=\frac{1}{2}(q-1)$,
these eigenvalues are mutually inverse.) A normal complement in $M_u$ to $M_u'=\F_p \ff$ would be a transitive subgroup
of order $pq$, and since $u\neq 0$, $q-1$, it would be as in Proposition \ref{mc-pq}(i) or (iv). However, none of these groups 
is normalised by $P$. Hence $M_u'$ does not have a normal complement in $M_u$. A similar argument applies to $\widehat{M}_u$.
\end{remark}

For $1 \leq u \leq \frac{1}{2}(q-1)$, we denote the isomorphism classes of the groups $M_u$ and $\widehat{M}_u$ by $\F_p^2\rtimes_u C_q$ 
and $\F_p^2\rtimes_u C_{2q}$.

\begin{proposition} \label{Aut-pq2}
The $2p$ groups $M$ of order $pq^2$ in Proposition \ref{mc-pq2} are
are all isomorphic as permutation groups, and 
$|\Aut(M,M')|=(p-1)(q-1)$ for these groups.

The same holds for the $2p$ groups of order $2pq^2$ in Proposition \ref{mc-2pq2}.
\end{proposition}
\begin{proof}
There is only one isomorphism class of abstract groups of the form $G=C_q \times (C_p \rtimes C_q)$.
Such a group contains $pq$ non-normal subgroups of order $q$, and these form a single orbit
under $\Aut(G)$. Hence all the groups of order $pq^2$ in Proposition \ref{mc-pq2} are 
isomorphic as permutation groups. A similar argument applies to the groups of order $2pq^2$. 

To find $|\Aut(M,M')|$, we may therefore suppose $M=\langle \be_1, T, A \rangle$ of order $p q^2$
or $M=\langle \be_1, T, A, B\rangle$ of order $2p q^2$.
In the first case, if $\theta \in \Aut(M,M')$ then $\theta(\be_1)=x
\be_1$ and $\theta(A)=A^y$ for some $x \in \F_p^\times$,
$y \in \F_q^\times$.  As $\theta(A) \theta(\be_1)
\theta(A)^{-1}=g\be_1$, we have $y=1$. Then 
$\theta(T)=[z \be_1, T^c A^d]$ with $z \in \F_p$, $c \in \F_q^\times$, $d
\in \F_q$. As $\theta(A)$ and $\theta(T)$ commute, we
have $gz=z$ and hence $z=0$. As $\theta(T) \theta(\be_1)
\theta(T)^{-1}=g\be_1$, we have $c+d=1$ in $\F_q$. Thus we
have $p-1$ choices for $x$ and $q-1$ choices for $c$, and these
choices determine $\theta$. Hence $|\Aut(M,M')|=(p-1)(q-1)$ for each
of the $2p$ groups of order $pq^2$.

Finally, if $M=\langle \be_1, T, A, B\rangle$ then the above automorphism $\theta$ on its subgroup of index $2$ extends to an element of $\Aut(M,M')$ 
by taking $\theta(B)=B$, and this extension is unique as $B$ is the
only element of order $2$ in $M'$.
\end{proof}

The groups of order $pq$ in Proposition \ref{mc-pq} are regular, so any two of them are 
isomorphic as permutation groups if they are isomorphic as abstract groups. We already know $\Aut(M)$ in these cases.

\begin{proposition} \label{Aut-2pq}
The $2p(q-1)$ groups in Proposition \ref{mc-2pq} of isomorphism type $C_ p \rtimes C_{2q}$ are all isomorphic as permutation groups. For these groups, 
$|\Aut(M,M')|=p-1$.

The $2p$ groups $M$ in Proposition \ref{mc-2pq} of isomorphism type $D_{2p}\times C_q$ are  
isomorphic as permutation groups. For these groups, $|\Aut(M,M')|=(p-1)(q-1)$.
\end{proposition}
\begin{proof}
In both cases, the stabiliser of $1_N$ has order $2$, and
all subgroups of order $2$ in $M$ are conjugate. Hence the groups of the same isomorphism type
as abstract groups are isomorphic as permutation groups. Moreover, $C_p \rtimes C_{2q}$ (respectively, $D_{2p} \times C_q$) has 
$p-1$ (respectively, $(p-1)(q-1)$) automorphisms fixing a given element of order $2$. 
\end{proof}

The information in Propositions \ref{Aut-p2q2}--\ref{Aut-2pq} is summarised in Table \ref{metab-trans-HGS}, where we have also used Lemma \ref{count-formula} to find the number of Hopf-Galois structures in each case.

\begin{table}[h]
\centerline{ 
\begin{tabular}{|c|c|c|c|c|c|c|c|} \hline
  Order & Structure & $\#$ groups & $|\Aut(M,M')|$ & $\#$ HGS \\ \hline
  $p^2 q^2$ & $N \rtimes (C_p \rtimes C_q)$ & $1$ & $2p(p-1)$ & $2$ \\ \hline
  $2p^2 q^2$ & $\Hol(N)$ & $1$ & $2p(p-1)$ & $2$ \\ \hline
  $p^2 q$ & $C_p \times (C_p \rtimes C_q)$  & $2$ & $p^2(p-1)$ & $2p$ \\
  & $\F_p^2 \rtimes_u C_q$, $1 \leq u \leq \frac{1}{2}(q-3)$ & $2$ & $p^2(p-1)$ & $2p$ \\
   & $\F_p^2 \rtimes_{\frac{1}{2}(q-1)} C_q$ & $1$ & $2p^2(p-1)$ & $2p$ \\ \hline
  $2p^2 q$ & $(C_p \times (C_p \rtimes C_q))\rtimes C_2$ & $2$ & $p(p-1)$ & $2$ \\
   & $\F_p^2 \rtimes_u C_{2q}$, $1 \leq u \leq \frac{1}{2}(q-3)$ & $2$ & $p(p-1)$ & $2$ \\
  & $\F_p^2 \rtimes_{\frac{1}{2}(q-1)} C_{2q}$ & $1$ & $2p(p-1)$ & $2$ \\ \hline
  $pq^2$ & $C_q \times (C_p \rtimes C_q)$ & $2p$ & $(p-1)(q-1)$ & $2(q-1)$ \\ \hline
  $2pq^2$ & $C_q \times (C_p \rtimes C_{2q})$ & $2p$ & $(p-1)(q-1)$ & $2(q-1)$ \\ \hline
  $pq$ &  $C_p \rtimes C_q$ & $2p(q-2)+2$ & $p(p-1)$ & $2p(q-2)+2$ \\
   & $C_{pq}$ & $2p$ & $(p-1)(q-1)$ & $2(q-1)$ \\ \hline
  $2pq$ &  $C_p \rtimes C_{2q}$ & $2p(q-1)$ & $p-1$ & $2(q-1)$ \\
   &  $D_{2p} \times C_q$ & $2p$ & $(p-1)(q-1)$ & $2(q-1)$ \\
\hline
\end{tabular}
}  
\vskip3mm

\caption{Hopf-Galois structures for $N$ metabelian} 
 \label{metab-trans-HGS}  	
\end{table}

We summarise our results for Hopf-Galois structures of non-abelian type in the following theorem.

\begin{theorem} \label{metab-summary}
There are in total $q+9$ isomorphism types of permutation groups $G$ of degree $pq$ which are realised
by Hopf-Galois structures of non-abelian type $C_p \rtimes C_q$, as listed in Table \ref{metab-trans-HGS}. These include the two regular groups, i.e.~the 
cyclic and non-abelian groups of order $pq$ (for which the corresponding Galois extensions have $2(q-1)$ and $2p(q-2)+2$ 
Hopf-Galois structures of non-abelian type respectively). For $q-1$ of these permutation groups (more precisely, for all but one of the groups of each of the orders $p^2 q$, $2p^2q$),
the corresponding field extensions fail to be almost classically Galois. In the remaining $10$ cases, the extensions are almost classically Galois.
\end{theorem}
\begin{proof}
Everything except the statements about almost classically Galois extensions follows from Table \ref{metab-trans-HGS}. For these,
see Remark \ref{Mu} for the groups of order $p^2 q$ and $2p^2 q$. In all the other cases, it is clear from the structure of $G$ that 
$G$ contains a regular normal subgroup.
\end{proof}

\begin{remark}
In general, any Hopf-Galois structure of non-abelian type has an ``opposite'' Hopf-Galois structure of the same type; see \cite[Lemma 2.4.2]{GP}.
This explains why the number of Hopf-Galois structures in each row of Table \ref{metab-trans-HGS} is even. 
\end{remark}

\subsection{Comparing the two types}

We can read off from Tables \ref{cyclic-abst} and \ref{metab-trans-HGS} which permutation groups $G$
are realised by Hopf-Galois structures of both types. The number of Hopf-Galois structures of each type are shown in Tables \ref{cyclic-HGS} and \ref{metab-trans-HGS}. This gives our final result. 

\begin{theorem}
There are six permutation groups $G$ of degree $n$ which are realised by Hopf-Galois structures of both cyclic and non-abelian types. These are 
as shown in Table \ref{both-types}. In all these cases, the corresponding field extensions are almost classically Galois.
\end{theorem}

The last two rows of Table \ref{both-types} recover the main results of \cite{pq}.

\begin{table}[h]
\centerline{ 
\begin{tabular}{|c|c|c|c|} \hline
  Order & Structure & $\#$ cyclic type HGS & $\#$ non-abelian type HGS \\ \hline
   $2pq^2$ & $C_q \times (C_p \rtimes C_{2q})$ & $1$ & $2(q-1)$ \\
     $pq^2$ & $C_q \times (C_p \rtimes C_q)$ & $1$ & $2(q-1)$ \\
    $2pq$ &  $C_p \rtimes C_{2q}$ & $1$ & $2(q-1)$ \\
    $2pq$ &  $D_{2p} \times C_q$ & $1$ & $2(q-1)$ \\
     $pq$ &  $C_p \rtimes C_q$ & $p$ & $2p(q-2)+2$ \\
  $pq$ & $C_{pq}$ & $1$ & $2(q-1)$ \\ \hline
\end{tabular}
}  
\vskip3mm

\caption{Groups admitting Hopf-Galois structures of both types} 
 \label{both-types}  	
\end{table}

 \section{Concluding Comments}
 
 We have obtained some general results on the permutation groups of squarefree degree $n$
 which are realised by Hopf-Galois structures, and we have listed all such groups in the case that
 $n=pq$ where $q$ is a Sophie Germain prime and $p$ is the associated safeprime. We have also determined
 those $G$ realised by Hopf-Galois structures of both possible types. In this special case,
 the field extensions admitting a Hopf-Galois structure of cyclic type are always almost classically Galois,
 whereas (for large $q$) most of those admitting a Hopf-Galois structure of non-abelian type are not.
 Moreover, we found no cases where two distinct permutation groups, both realised by Hopf-Galois structures, had the same underlying abstract group.
 It would be interesting to investigate whether similar behaviour occurs for more general squarefree degrees $n$. 
 
 \section*{Acknowledgement}
 Much of the work on this paper was done during the summer of 2019, when the second-named author held London
   Mathematical Society Undergraduate Research Bursary URB-18-19-17.
     We thank the London Mathematical Society and the Department of Mathematics 
   of the University of Exeter for supporting this bursary.                        

\bibliography{HGS-SG}

\end{document}